\documentclass[11pt]{amsart}
\author{Christian J. Berghoff}
\title{Universal elliptic Gau{\ss} sums\\ for Atkin primes in Schoof's algorithm}
\address{Universit\"at Bonn, Mathematisches Institut, Endenicher Allee 60, 53115 Bonn, Germany}
\email{berghoff@math.uni-bonn.de}

\usepackage{lmodern}
\usepackage[T1]{fontenc}
\usepackage[utf8x]{inputenc}
\usepackage{a4wide}
\usepackage[english]{babel}
\usepackage{amsmath, amssymb, amsthm, amsfonts}
\usepackage{framed}
\usepackage{enumerate}
\usepackage{url}

\usepackage{algorithmic}

\algsetup{linenodelimiter=.}

\makeatletter
\newcounter{algorithm}
\renewcommand{\thealgorithm}{\arabic{algorithm}}
\def\algorithm{\@ifnextchar[{\@algorithma}{\@algorithmb}}
\def\@algorithma[#1]{%
	\refstepcounter{algorithm}
	\trivlist
	\leftmargin\z@
	\itemindent\z@
	\labelsep\z@
	\item[\parbox{\columnwidth}{%
		\hrule
		\hrule
		\noindent\strut\textbf{Algorithm \thealgorithm.} #1
		\hrule
	}]\hfil\vskip0em%
}
\def\@algorithmb{\@algorithma[]}

\makeatother

\usepackage{mathabx}
\usepackage[all]{xy}
\usepackage{textcomp}
\usepackage{stmaryrd}
\usepackage{nicefrac}
\usepackage[numbers]{natbib}
\usepackage[pdftex]{graphicx}
\usepackage{layout}
\usepackage{fancyhdr}
\usepackage{setspace}
\usepackage{color}
\usepackage{hyperref}

\usepackage{float}
\usepackage{pgfplots}
\pgfplotsset{compat=1.8}
\usepackage{subfigure}

\numberwithin{equation}{section}

\theoremstyle{plain}

\newtheorem{satz}{Theorem}[section]

\newtheorem{lemma}[satz]{Lemma}
\newtheorem{koro}[satz]{Corollary}

\theoremstyle{definition}

\theoremstyle{remark}
\newtheorem{bem}[satz]{Remark}

\DeclareMathOperator{\gal}{Gal}

\DeclareMathOperator{\Tr}{Tr}
\DeclareMathOperator{\deg2}{deg}
\DeclareMathOperator{\ord}{ord}

\DeclareMathOperator{\SL}{SL}

\let\prec\relax 
\DeclareMathOperator{\prec}{prec}

\newcommand{\OO}{\mathcal{O}}
\newcommand{\Q}{\mathbb{Q}}
\newcommand{\F}{\mathbb{F}}

\newcommand{\N}{\mathbb{N}}
\newcommand{\Z}{\mathbb{Z}}

\newcommand{\C}{\mathbb{C}}

\newcommand{\An}{\mathbf{A}}

\begin{document}
\begin{abstract}
This work builds on the results obtained in \cite{ELGS_Rechnungen, ELGS}. We define universal elliptic Gau{\ss} sums for Atkin primes in Schoof's algorithm for counting points on elliptic curves. Subsequently, we show these quantities admit an efficiently computable representation in terms of the $j$-invariant and two other modular functions. We analyse the necessary computations in detail and derive an alternative approach for determining the trace of the Frobenius homomorphism for Atkin primes using these pre-computations. A rough run-time analysis shows, however, that this new method is not competitive with existing ones.
\end{abstract}
\maketitle

\tableofcontents
\section{Elliptic curves}\label{sec:ell_kurven}
We consider primes $p>3$ and thus assume that the curve $E/\F_p$ in question is given in the Weierstra{\ss} form
\[ 
E: Y^2=X^3+aX+b=f(X),
 \]
where $a, b \in \F_p$. For the following well-known statements cf. \cite{Silverman, Washington}. We assume that the elliptic curve is neither singular nor supersingular. It is a standard fact that $E$ is an abelian group with respect to point addition. Its neutral element, the point at infinity, will be denoted $\OO$. For a prime $\ell \neq p$, the $\ell$-torsion subgroup $E[\ell]$ has the shape
\[ 
E[\ell]\cong \Z/\ell\Z \times \Z/\ell\Z.
 \]
The Weil pairing $e_\ell: E[\ell] \times E[\ell] \rightarrow \mu_\ell$ is defined on the $\ell$-torsion. In the endomorphism ring of $E$ the Frobenius homomorphism
\[ 
\phi_p: (X, Y) \mapsto (\varphi_p(X), \varphi_p(Y))=(X^p, Y^p)
 \]
satisfies the quadratic equation 
\begin{equation}\label{eq:char_gl}
0=\chi(\phi_p)=\phi_p^2-t\phi_p+p,
 \end{equation} 
where $|t|\leq 2\sqrt{p}$ by the Hasse bound. By restriction $\phi_p$ acts as a linear map on $E[\ell]$. The number of points on $E$ over $\F_p$ is given by $\#E(\F_p)=p+1-t$ and is thus immediate from the value of $t$.\\
Schoof's algorithm computes the value of $t$ modulo $\ell$ for sufficiently many small primes $\ell$ by considering $\chi(\phi_p)$ modulo $\ell$ and afterwards combines the results by means of the Chinese Remainder Theorem. In the original version this requires computations in extensions of degree $\OO(\ell^2)$.
However, a lot of work has been put into elaborating improvements. Let $\Delta=t^2-4p$ denote the discriminant of equation \eqref{eq:char_gl}. Then we distinguish the following cases:
\begin{enumerate}
\item If $\left( \frac{\Delta}{\ell}\right)=1$, then $\ell$ is called an \textit{Elkies prime}. In this case, the characteristic equation factors as $\chi(\phi_p)=(\phi_p-\lambda)(\phi_p-\mu) \mod \ell$ and finding the value of $t$ modulo $\ell$ only requires working in extensions of degree $\OO(\ell)$.
\item If  $\left( \frac{\Delta}{\ell}\right)=-1$, then $\ell$ is called an \textit{Atkin prime}. In this case the eigenvalues of $\phi_p$ are in $\F_{\ell^2}\backslash\F_\ell$ and there is no eigenpoint $P \in E[\ell]$. This is the case for which a new approach is presented in this paper.
\end{enumerate}
\section{Modular functions}
We recall some facts from \cite{ELGS}, where further details may be found. A modular function of weight $k \in \Z$ for a subgroup $\Gamma^\prime \subseteq \textup{SL}_2(\Z)=:\Gamma$ is a meromorphic function $f(\tau)$ on the upper complex half-plane $\mathbb{H}=\{\tau \in \C: \Im(\tau)>0\}$ satisfying
\begin{equation}\label{eq:def_mod_funk}
f(\gamma\tau)=(c\tau+d)^kf(\tau)\ \textup{for all}\ \gamma=\left(\begin{smallmatrix}
a & b \\ c & d
\end{smallmatrix}\right) \in \Gamma^\prime,
\end{equation}
where $\gamma\tau=\frac{a\tau+b}{c\tau+d}$, and some technical conditions. Equation \eqref{eq:def_mod_funk} in particular implies $f$ can be written as a Laurent series in terms of $q_N=\exp\left(\frac{2\pi i\tau}{N}\right)$ for some $N \in \N$ depending on $\Gamma^\prime$. We use the notation $q=q_1$ and consider the groups 
$\Gamma^\prime=\Gamma_0(\ell):=\left\{\left(\begin{smallmatrix}
a & b\\ c& d
\end{smallmatrix}\right) \in \textup{SL}_2(\Z): \ell \mid c \right\}$. The field of modular functions of weight $0$ for a group $\Gamma^\prime$ will be denoted by $\An_0(\Gamma^\prime)$ and the subfield of holomorphic functions by $\mathbf{H}_0(\Gamma^\prime)$. The Fricke-Atkin-Lehner involution $w_\ell$ acts on modular functions $f(\tau)$ via $f(\tau) \mapsto f\left(\frac{-1}{\ell\tau}\right)=:f^*(\tau)$, where $f^*(\tau)=f(\ell\tau)$ for $f(\tau) \in \An_0(\textup{SL}_2(\Z))$ holds.\\
Denoting by $\wp(z, \tau)$ the Weierstra{\ss} $\wp$-function and putting $w=e^{2\pi i z}$, for $|q|<|w|<|q^{-1}|$ one obtains the following equations:
\begin{align}
\frac{1}{(2\pi i)^2}\wp(z, \tau)&=\frac{1}{12}-2\sum_{n=1}^{\infty}\frac{q^n}{(1-q^n)^2} +\sum_{n \in \Z} \frac{q^n w}{(1-q^n w)^2}=:x(w, q), \label{eq:x_def}\\
\frac{1}{(2\pi i)^3 }\wp^{\prime}(z, \tau)&=\sum_{n \in \Z}\frac{q^n w(1+q^n w)}{(1-q^n w)^3}=:2y(w, q).
\label{eq:y_def}
\end{align}
The Tate curve $E_q$ is given by the equation
\[ 
y(w, q)^2=x(w, q)^3-\frac{E_4(q)}{48}x(w, q)+\frac{E_6(q)}{864},
 \]
where $E_4, E_6$ are modular functions for $\Gamma$ of weight $4$ and $6$, respectively.
Next, we recall the Laurent series
\begin{align}
\eta(q)&=q^{\frac{1}{24}}\left(1+\sum_{n=1}^{\infty}(-1)^n\left(q^{n(3n-1)/2}+q^{n(3n+1)/2}\right)\right),\label{eq:eta_def}\\
m_\ell(q)&=\ell^s \left(\frac{\eta(q^\ell)}{\eta(q)}\right)^{2s}\quad \text{with}\quad s=\frac{12}{\gcd(12, \ell-1)}.\label{eq:ml_def}
\end{align}
We further use $p_1(q)=\sum_{\zeta \in \mu_\ell, \zeta \neq 1} x(\zeta, q)$, the modular discriminant $\Delta(q)=\eta(q)^{24}$ and the well-known $j$-invariant $j(q)$. $p_1$ is a modular function of weight $2$ for $\Gamma_0(\ell)$, whereas $\Delta$ and $j$ are modular functions for $\Gamma$ of weight $12$ and $0$, respectively. Furthermore, $\Delta(\tau)\neq 0$ holds for $\tau \in \mathbb{H}$.\\
There exists a polynomial $M_\ell \in \C[X, Y]$, sometimes referred to as the canonical modular polynomial, such that $M_\ell(X, j(q))$ is irreducible over $\C(j(q))[X]$ and $m_\ell(q)$ is one of its roots.\\
Finally, we recall some statements on fields of modular functions of weight $0$, which were shown in \cite{ELGS}.
\begin{lemma}\label{lemma:untergruppe_gal_gruppe}
Let $\Gamma^{\prime}$ be a subgroup of $\Gamma$ with $\left\{
\left( \begin{smallmatrix} a & b\\ c& d\end{smallmatrix}\right)
\equiv 
\left(\begin{smallmatrix} 1 & 0\\ 0 & 1\end{smallmatrix}\right) 
\mod N
\right\}=:\Gamma(N)\leq \Gamma^{\prime}$, then
\[ 
\gal(\mathbf{A}_0(\Gamma(N))/\mathbf{A}_0(\Gamma^{\prime}))\cong (\pm\Gamma^{\prime})/(\pm\Gamma(N))
 \]
holds. In particular $\mathbf{A}_0(\Gamma^{\prime})$ is a finite extension of $\mathbf{A}_0(\Gamma)=\C(j)$ of degree $[\Gamma: \pm \Gamma^{\prime}]$.
\end{lemma}

\begin{satz}\label{satz:j_f_erzeugen_mod_funk}
Let $f(\tau) \in \mathbf{A}_0(\Gamma_0(\ell))\backslash\mathbf{A}_0(\Gamma)$ be a modular function of weight $0$ for $\Gamma_0(\ell)$, but not for $\Gamma$. Then
\[ 
\mathbf{A}_0(\Gamma_0(\ell))=\mathbf{A}_0(\Gamma)(f(\tau))=\C(f(\tau), j(\tau))  
 \]
holds.
\end{satz}
The following lemma from \cite[pp.~206--208]{Neukirch} will be used below.
\begin{lemma}\label{lem:neukirch-lemma}
Let $L/K$ be an extension of fields, $\OO \subseteq K$ be a ring. Let $\alpha \in L\backslash K$ have the minimal polynomial $f(X) \in K[X]$ of degree $n$. Then the $\OO$-module 
\[ 
C_{\alpha}=\{x \in L\mid \Tr_{L/K}(x\OO[\alpha]) \subseteq \OO\}
 \] 
has the $\OO$-basis 
\[ 
\left\{\frac{\alpha^i}{f^{\prime}(\alpha)}, i=0, \ldots, n-1\right\}.
 \]
\end{lemma}
As a corollary one obtains
\begin{koro}\cite[Proposition 2.16]{ELGS}\label{koro:darstellung_besser}
Let $g(\tau) \in \mathbf{H}_0(\Gamma_0(\ell))\backslash\mathbf{H}_0(\Gamma)$ be a holomorphic modular function. Then $g(\tau)$ admits the following representation:

\[ 
g(\tau)=\frac{Q(m_\ell(\tau), j(\tau))}{m_\ell(\tau)^k\frac{\partial M_\ell}{\partial Y}(m_\ell(\tau), j(\tau))}
 \]
for a $k\geq 0$ and a polynomial $Q(X, Y) \in \C[X, Y]$ with $\deg2_Y(Q)<v=\deg2_Y(M_\ell)$.
\end{koro}
\section{Universal elliptic Gau{\ss} sums in the Atkin case}

We show how to use universal elliptic Gau{\ss} sums for computing the value of the trace $t$ of the Frobenius homomorphism modulo Atkin primes $\ell$.

\subsection{Definition}
We begin with the following result
\begin{lemma}\label{lem:wp_trafo_allgemein}
Let $\tau \in \mathbb{H}$, $v_1, v_2 \in \Z$ and $\gamma=\begin{pmatrix}
a & b \\ c & d
\end{pmatrix}
\in \SL_2(\Z)$. Then we have
\begin{align}
\left.\wp\left(\frac{v_1\tau+v_2}{\ell}, \tau\right)\right|_{\gamma}=&(c\tau+d)^2\wp\left(\frac{v_1(a\tau+b)+v_2(c\tau+d)}{\ell}, \tau\right),\\
\left.\wp^{\prime }\left(\frac{v_1\tau+v_2}{\ell}, \tau\right)\right|_{\gamma}=&(c\tau+d)^3\wp^{\prime}\left(\frac{v_1(a\tau+b)+v_2(c\tau+d)}{\ell}, \tau\right).
\end{align}
\end{lemma}
\begin{proof}
For $z \in \C$ we first compute
\begin{align*}
\wp\left(z, \frac{a\tau+b}{c\tau+d}\right)&=\frac{1}{z^2}+\sum_{n^2+m^2\neq 0}\left(\frac{1}{(z-(m+n\frac{a\tau+b}{c\tau+d}))^2} -\frac{1}{(m+n\frac{a\tau+b}{c\tau+d})^2} \right)\\
&=(c\tau+d)^2\cdot \frac{1}{((c\tau+d)z)^2}\\
&+(c\tau+d)^2\sum_{m^2+n^2\neq 0}\left( \frac{1}{((c\tau+d)z-S_{a,b,c,d}(m, n))^2} -\frac{1}{(S_{a,b,c,d}(m, n))^2} \right)\\
&=(c\tau+d)^2\wp((c\tau+d)z, \tau),
\end{align*}
where we make use of the abbreviation $S_{a,b,c,d}(m, n)=m(c\tau+d)+n(a\tau+b)$ and the last equation follows from $ad-bc=1$. From this we obtain
\begin{align*}
\left.\wp\left(\frac{v_1\tau+v_2}{\ell}, \tau\right)\right|_{\gamma}
=\wp\left(\frac{v_1\frac{a\tau+b}{c\tau+d}+v_2}{\ell}, \frac{a\tau+b}{c\tau+d}\right)
=(c\tau+d)^2\wp\left(\frac{v_1(a\tau+b)+v_2(c\tau+d)}{\ell}, \tau\right).
\end{align*}
The proof for the derivative of $\wp$ proceeds analogously.
\end{proof}

Using this statement we will construct a modular function for the group $\Gamma_0^0(\ell)$ we now define.
\begin{lemma}\label{lem:Gamma_0_0_def}
Let $\ell$ be a prime. Then
\[ 
\Gamma_0^0(\ell)=
\left\{\begin{pmatrix}
a & b\\
c & d\\
\end{pmatrix}: b\equiv c\equiv 0 \mod \ell \right\} \subseteq \Gamma
 \]
is a subgroup of $\Gamma$ and we have the inclusions $\Gamma(\ell)\subseteq \Gamma_0^0(\ell) \subseteq \Gamma_0(\ell)$. A system of representatives for $\Gamma/\Gamma_0^0(\ell)$ is constituted by
\[ 
S_{\lambda, k}=\begin{pmatrix}
\lambda & -1+\lambda k \\ 1 & k
\end{pmatrix}\quad \text{for}\quad 0 \leq \lambda, k < \ell,\quad
S_{\lambda, \ell}=\begin{pmatrix}
1 & \lambda \\
0 & 1
\end{pmatrix}
\quad \text{for}\quad 0\leq \lambda <\ell.
 \]
\end{lemma}
\begin{proof}
The fact that $\Gamma_0^0(\ell)$ is a group and the inclusions are obvious. To prove the correctness of the system of representatives we remark that the matrices
\begin{equation}\label{eq:def_matrix_T}
T_{\lambda}=\begin{pmatrix}
1 & \lambda \\ 0 & 1
\end{pmatrix}\quad \text{for}\quad 0 \leq \lambda < \ell
 \end{equation}
represent the cosets $\Gamma_0(\ell)/\Gamma_0^0(\ell)$, as is easily seen. Multiplying the $T_{\lambda}$ by the matrices 
\begin{equation}\label{eq:def_rep_Sk}
S_k=\begin{pmatrix}
0 & -1 \\ 1 & k
\end{pmatrix}\quad \text{for}\quad 0\leq k< \ell,\quad
S_\ell=\begin{pmatrix}
1 & 0 \\ 0 & 1
\end{pmatrix}
 \end{equation}
from \cite[p.~54]{Mueller} that form a system of representatives for $\Gamma/\Gamma_0(\ell)$ one obtains the claim.
\end{proof}


\begin{koro}\label{koro:gs_atkin_def}
Let $\ell$ be a prime, $n \mid \ell-1$ and $\chi: \F_\ell^* \rightarrow \mu_n$ be a character of order $n$. Let the $\ell$-th root of unity $\xi$ be the image under the Weil pairing $e_\ell$ of the $\ell$-torsion points $P=(x(\zeta_\ell, q), y(\zeta_\ell, q))$, $Q_0=(x(q^{\frac{1}{\ell}}, q), y(q^{\frac{1}{\ell}}, q))$ on the Tate curve, and let $G_{\chi^{-1}}(\xi)=\sum_{\lambda \in \F_\ell^*}\chi^{-1}(\lambda)\xi^{\lambda}$ be the corresponding cyclotomic Gau{\ss} sum. Define
\[ 
G_{\ell, n, \chi}(q)=\sum_{\lambda \in \F_\ell^*} \chi(\lambda)V(\zeta_\ell^{\lambda}, q)=\sum_{\lambda \in \F_\ell^*} \chi(\lambda)(\lambda P)_V 
\quad\text{for}\quad V=
\begin{cases}
x,\quad n\equiv 1 \mod 2,\\
y, \quad n\equiv 0 \mod 2.
\end{cases}
 \]
and
\[
H_{\ell, n, \chi}(q)=\sum_{\lambda \in \F_\ell^*} \chi(\lambda)V(q^{\frac{\lambda}{\ell}}, q)=\sum_{\lambda \in \F_\ell^*}\chi(\lambda)(\lambda Q_0)_V.
 \]
Then the function 
\[ 
\sigma_{\ell, n, \chi}(q)=\frac{G_{\ell, n, \chi}(q)H_{\ell, n, \chi}(q)p_1(q)^rG_{\chi^{-1}}(\xi)}{\Delta(q)}\quad\text{for}\quad
r=\begin{cases}
4, \quad n \equiv 1 \mod 2,\\
3, \quad n \equiv 0 \mod 2,
\end{cases}
 \]
which we will call a \textup{universal elliptic Gau{\ss} sum (for Atkin primes)},
exhibits the following properties:
\begin{enumerate}
\item $\sigma_{\ell, n, \chi}(q)$ is a modular function of weight $0$ for the group $\Gamma_0^0(\ell)$.
\item $\sigma_{\ell, n, \chi}(q)$ is holomorphic on $\mathbb{H}$.
\item $\sigma_{\ell, n, \chi}(q)$ is invariant under transformations of the form $(P, Q_0) \mapsto (aP, bQ_0)$ for values $a, b \in \F_\ell^*$.
\item $\sigma_{\ell, n, \chi}(q)$ has coefficients in $\Q[\zeta_n]$.
\end{enumerate}
\end{koro}
\begin{proof}
\begin{enumerate}
\item Using the lemma \ref{lem:wp_trafo_allgemein} we have just shown for $\gamma=
\left( \begin{smallmatrix}
a & b \\
c & d
\end{smallmatrix}\right)
 \in \Gamma_0^0(\ell)$ and $v_1 \in \F_\ell^*$ we calculate
\[ 
\left.x\left(q^{\frac{v_1}{\ell}}, q\right)\right|_\gamma
=\left.k\wp\left(\frac{v_1\tau}{\ell}, \tau\right)\right|_{\gamma}
=k(c\tau+d)^2\wp\left(\frac{v_1(a\tau+b)}{\ell}, \tau\right)
=(c\tau+d)^2x\left(q^{\frac{v_1a}{\ell}}, q\right).
\]
We employ formulae \eqref{eq:x_def}, \eqref{eq:y_def} (with $k=\frac{1}{(2\pi i)^2})$ for the transition to the Weierstra{\ss} $\wp$-function (nota bene that $|q|<|q^{\frac{v_1}{\ell}}|<1$ holds) and use $b\equiv 0 \mod \ell$ in the last step. In an analogue way one shows
\[ 
\left.y\left(q^{\frac{v_1}{\ell}}, q\right)\right|_\gamma=(c\tau+d)^3y\left(q^{\frac{v_1a}{\ell}}, q\right).
 \]
This in turn allows to determine the transformation behaviour of $H_{\ell, n, \chi}(q)$ under action of $\gamma$, viz.
\begin{align*} 
\left.H_{\ell, n, \chi}(q)\right|_{\gamma}=&\sum_{\lambda \in \F_{\ell}^*}\chi(\lambda)\left.V\left(q^{\frac{\lambda}{\ell}}, q\right)\right|_{\gamma}
=(c\tau+d)^e\sum_{\lambda \in \F_{\ell}^*}\chi(\lambda)V\left(q^{\frac{\lambda a}{\ell}}, q\right)\\
=&(c\tau+d)^e\chi^{-1}(a)H_{\ell, n, \chi}(q),
 \end{align*}
where $e=2$ holds for $n\equiv 1 \mod 2$ and $e=3$ otherwise. Similarly, one can show $\left.G_{\ell, n, \chi}(q)\right|_{\gamma}=(c\tau+d)^e\chi^{-1}(d)G_{\ell, n, \chi}(q)$. Using our knowledge of the transformation behaviour of the remaining functions occurring in the definition of $\sigma_{\ell, n, \chi}(q)$ we finally obtain
\[ 
\left.\sigma_{\ell, n, \chi}(q)\right|_{\gamma}
=\sigma_{\ell, n, \chi}(q)\frac{(c\tau+d)^{2e+2r}\chi^{-1}(d)\chi^{-1}(a)}{(c\tau+d)^{12}}
=\sigma_{\ell, n, \chi}(q)\chi^{-1}(ad)
=\sigma_{\ell, n, \chi}(q).
 \]
In the last step we have used $ad \equiv 1 \mod \ell$, which follows from $b\equiv c\equiv 0 \mod \ell$.\\
It remains to show that $\sigma_{\ell, n, \chi}(q)$ is meromorphic at the cusps, i.~e. that the Fourier expansion of $\left.\sigma_{\ell, n, \chi}(q)\right|_{S_{\lambda, k}}$ with the matrices $S_{\lambda, k}$ from lemma \ref{lem:Gamma_0_0_def} contains only finitely many negative exponents. However, for $\gamma=\left( \begin{smallmatrix}
a & b\\ c & d
\end{smallmatrix}\right) \in \Gamma$ lemma \ref{lem:wp_trafo_allgemein} implies
\[ 
\left.x\left(q^\frac{v_1}{\ell}\zeta_{\ell}^{v_2}, q\right)\right|_{\gamma}=(c\tau+d)^2x\left(q^{\frac{v_1a+v_2c}{\ell}}\zeta_{\ell}^{v_1b+v_2d}, q\right)
\]
and an analogue statement holds for $y(\cdot, q)$. When computing the arising expression using formulae \eqref{eq:x_def} and \eqref{eq:y_def}, respectively, one immediately sees the $q$-expansion contains only finitely many negative exponents. From the definition of these expressions the same holds true for $\left.G_{\ell, n, \chi}(q)\right|_{S_{\lambda, k}}$, $\left.H_{\ell, n, \chi}(q)\right|_{S_{\lambda, k}}$ and $\left.p_1(q)\right|_{S_{\lambda, k}}$ and, finally, for $\left.\sigma_{\ell, n, \chi}(q)\right|_{S_{\lambda, k}}$.

\item Formula \eqref{eq:x_def} directly implies that $x(q^{\frac{v}{\ell}}, q)$ has poles if and only if $q^n=1$ or $q^n=q^{\frac{v}{\ell}}$ holds for $n \in \Z$. However, $\tau \in \mathbb{H}$ yields $|q|<1$, whereas the second equality cannot hold for $0<v<\ell$. Hence, $H_{\ell, n, \chi}(q)$ is holomorphic on $\mathbb{H}$ by construction. In the same way one sees that $x(\zeta_\ell^v, q)$ is holomorphic. Using our knowledge about the remaining functions to be considered the claim is immediate.
\item First, $(aP, bQ_0)=((x, y)(\zeta_{\ell}^a, q), (x, y)(q^{\frac{b}{\ell}}, q))$ holds. Obviously both $\Delta(q)$ and $p_1(q)$ are invariant under this transformation. Furthermore, one obtains
\[ 
\sum_{\lambda \in \F_{\ell}^*} \chi(\lambda)V(q^{\frac{b\lambda}{\ell}}, q)
=\chi^{-1}(b)\sum_{\lambda \in \F_{\ell}^*} \chi(b\lambda)V(q^{\frac{b\lambda}{\ell}}, q)
=\chi^{-1}(b)H_{\ell, n, \chi}(q)
 \]
and in the same vein $\sum_{\lambda \in \F_{\ell}^*} \chi(\lambda)V(\zeta_{\ell}^{a\lambda}, q)=\chi^{-1}(a)G_{\ell, n, \chi}(q)$. In addition the term $G_{\chi^{-1}}(\xi)$ transforms according to 
\[ 
G_{\chi^{-1}}(e_{\ell}(aP, bQ_0))=G_{\chi^{-1}}(e_{\ell}(P, Q_0)^{ab})=G_{\chi^{-1}}(\xi^{ab})=\chi(ab)G_{\chi^{-1}}(\xi),
 \]
as follows from the properties of the Weil pairing and cyclotomic Gau{\ss} sums. By multiplying the arising factors we see that $\sigma_{\ell, n, \chi}(q)$ is invariant under the transformation.
\item This follows from what we have already proven. Obviously the coefficients of $\sigma_{\ell, n, \chi}(q)$ lie in $\Q[\zeta_{\ell}, \zeta_n]$. We choose a generator $c$ of $\F_{\ell}^*$ and consider the action of the homomorphism $\sigma: \zeta_{\ell} \mapsto \zeta_{\ell}^c$ generating $\gal(\Q[\zeta_{\ell}, \zeta_n]/\Q[\zeta_n])$ on $\sigma_{\ell, n, \chi}(q)$. We calculate
\[ 
\sigma(G_{\ell, n, \chi}(q))=\sum_{\lambda \in \F_\ell^*} \chi(\lambda)\sigma(V(\zeta_\ell^{\lambda}, q))
=\sum_{\lambda \in \F_\ell^*} \chi(\lambda)V(\zeta_\ell^{c\lambda}, q)
=\chi^{-1}(c)G_{\ell, n, \chi}(q)
 \]
and in the same way one obtains
\[ 
\sigma(G_{\chi^{-1}}(\xi))=\chi(c)G_{\chi^{-1}}(\xi).
 \]
Since the remaining terms in the definition of $\sigma_{\ell, n, \chi}(q)$ are invariant under $\sigma$, the coefficients of $\sigma_{\ell, n, \chi}(q)$ lie in the fixed field of this homomorphism.
\end{enumerate}
\end{proof}


\subsection{Rational representation}\label{sec:sigma_darstellung}

We now wish to represent the expression $\sigma_{\ell, n}(q)$ in terms of $j(\tau)$ and other modular functions. Actually, our representation will depend on two other modular functions apart from the $j$-invariant. This results from the following
\begin{lemma}\label{lem:erzeuger_koerper_g_0_0}
Let $f(\tau) \in \mathbf{A}_0(\Gamma_0(\ell))\backslash \mathbf{A}_0(\Gamma)$ be a modular function of weight $0$ and let  the matrix $S_0=\left(\begin{smallmatrix}
0 & -1 \\ 1 & 0
\end{smallmatrix}\right)$ be as in \eqref{eq:def_rep_Sk}. Then we obtain
\begin{equation}\label{eq:koerper_gamma_0_0}
\mathbf{A}_0(\Gamma_0^0(\ell))=\C(j(\tau), f(\tau), f(S_0\tau)).
 \end{equation}
\end{lemma}
\begin{proof}
We first remark that
\begin{equation}\label{eq:gamma_0_oben}
\Gamma_0^0(\ell)=\Gamma_0(\ell)\cap \Gamma^0(\ell)\quad \text{with}\quad \Gamma^0(\ell)=\left\{ \begin{pmatrix}
a & b \\ c & d
\end{pmatrix}: b \equiv 0 \mod \ell  \right\}
 \end{equation}
holds. Now lemma \ref{lemma:untergruppe_gal_gruppe}) yields
\[ 
\gal(\mathbf{A}_0(\Gamma(\ell))/\mathbf{A}_0(\Gamma_0^0(\ell)))=\gal(\mathbf{A}_0(\Gamma(\ell))/\mathbf{A}_0(\Gamma_0(\ell)))\cap \gal(\mathbf{A}_0(\Gamma(\ell))/\mathbf{A}_0(\Gamma^0(\ell))),
 \]
which directly implies
\begin{equation}\label{eq:koerper_kompositum}
\mathbf{A}_0(\Gamma_0^0(\ell))=\mathbf{A}_0(\Gamma_0(\ell))\mathbf{A}_0(\Gamma^0(\ell)).
 \end{equation}
Furthermore, for $a, b, c, d \in \Z$ we compute
\[
S_0^{-1}
\begin{pmatrix}
a & b \\ c & d
\end{pmatrix}
S_0
=\begin{pmatrix}
d & -c \\ -b & a
\end{pmatrix},
 \]
which yields $\Gamma^0(\ell)=S_0^{-1}\Gamma_0(\ell)S_0$. From this we deduce that the function $f(S_0\tau)$ lies in $\mathbf{A}_0(\Gamma^0(\ell))$. Namely, writing $\gamma \in \Gamma^0(\ell)$ as $\gamma=S_0^{-1}\gamma^\prime S_0$ for some $\gamma^{\prime} \in \Gamma_0(\ell)$, we obtain
\begin{equation}\label{eq:wechsel_gamma_0_gamma_0_oben}
f(S_0\gamma\tau)=f(\gamma^{\prime}S_0\tau)=f(S_0\tau).
 \end{equation}
According to theorem \ref{satz:j_f_erzeugen_mod_funk} $\mathbf{A}_0(\Gamma_0(\ell))=\C(j(\tau), f(\tau))$ holds. Adapting its proof we directly see $\mathbf{A}_0(\Gamma^0(\ell))=\C(j(\tau), f(S_0\tau))$, which yields the claim together with equation \eqref{eq:koerper_kompositum}.
\end{proof}

Furthermore, we obtain
\begin{lemma}\label{lem:darst_sigma_allgemein}
Let $g(\tau) \in \mathbf{H}_0(\Gamma_0^0(\ell))\backslash\mathbf{H}_0(\Gamma_0(\ell))$ be a holomorphic modular function and let $k(\tau) \in \mathbf{H}_0(\Gamma_0^0(\ell))\backslash\mathbf{H}_0(\Gamma_0(\ell))$ with minimal polynomial $Q_k(X)$, $\deg2(Q_k)=\ell$, over $\mathbf{A}_0(\Gamma_0(\ell))$. Then $g(\tau)$ admits a representation of the form
\begin{equation}\label{eq:darst_sigma_1}
g(\tau)=\frac{\sum_{i=0}^{\ell-1}a_i k(\tau)^i}{\frac{\partial Q_k}{\partial X}(k(\tau))}
 \end{equation}
with $a_i \in \mathbf{H}_0(\Gamma_0(\ell))$.
\end{lemma}
\begin{proof}
We use lemma \ref{lem:neukirch-lemma}. We set $K=\mathbf{A}_0(\Gamma_0(\ell))$, $L=\mathbf{A}_0(\Gamma_0^0(\ell))$ and $\alpha=k(\tau)$ with minimal polynomial $f(X)=Q_k(X)$. Furthermore, let
\[ 
\OO=\{h(\tau) \in K: h(\tau)\text{ holomorphic}\}=\mathbf{H}_0(\Gamma_0(\ell)).
 \]
Since $k$ is assumed to be holomorphic this holds for all elements $z \in \OO[k]$ and thus also for $g(\tau)z$ and $\Tr_{L/K}(g(\tau)z)$. This yields $g(\tau) \in C_k$ (cf. lemma \ref{lem:neukirch-lemma}), which implies the claim.
\end{proof}

The preceding lemmas suggest the following procedure for computing an expression for $\sigma_{\ell, n}$: If we choose a holomorphic modular function $f(\tau) \in \mathbf{H}_0(\Gamma_0(\ell))\backslash\mathbf{H}_0(\Gamma)$ the proof of lemma \ref{lem:erzeuger_koerper_g_0_0} shows that $f(S_0\tau) \in \mathbf{H}_0(\Gamma_0^0(\ell))\backslash\mathbf{H}_0(\Gamma_0(\ell))$ holds. If we know the minimal polynomial $Q_{f, S_0}(X)$ of $f(S_0\tau)$ over $\mathbf{A}_0(\Gamma_0(\ell))$, we can thus first determine a representation for $\sigma_{\ell, n}$ by using lemma \ref{lem:darst_sigma_allgemein}. Since the coefficients $a_i$ that occur lie in $\mathbf{H}_0(\Gamma_0(\ell))$, according to corollary \ref{koro:darstellung_besser} one can subsequently determine a representation in terms of $j(\tau)$ and $f(\tau)$ for each of these coefficients. Combining everything one obtains the representation in terms of $j(\tau)$, $f(\tau)$ and $f(S_0\tau)$ that we will we denote by
\begin{equation}\label{eq:sigma_darst_allgemein}
\sigma_{\ell, n, \chi}(q)=R(j(\tau), f(\tau), f(S_0\tau)).
 \end{equation}
We have to specify up to which precision the Laurent series of the various modular functions have to be computed in order to derive the representations. We examine this for the choice $f(\tau)=m_\ell(\tau)$, using the notation $m_{\ell, 2}(\tau):=m_\ell(S_0\tau)$.\\


We first need some statements on the transformation behaviour of the $\eta$-function.
\begin{satz}\cite[p.~113, 126, 130]{Weber}\label{satz:eta_trafo}
Under action from $\gamma=\left(\begin{smallmatrix}
a & b\\ c & d
\end{smallmatrix} \right) \in \SL_2(\Z)$ the $\eta$-function transforms via
\[ 
\eta(\gamma\tau)=\varepsilon \cdot \sqrt{c\tau+d}\cdot\eta(\tau).
 \]
The values for the $24$-th root of unity $\varepsilon$ can be computed according to
\[ 
\varepsilon=\begin{cases}
\left( \frac{c}{d}\right)i^{(d-1)/2}\exp\left(\frac{\pi i}{12}(d(b-c)-(d^2-1)ac) \right),
& d\equiv 1 \mod 2, d>0,\\
\left( \frac{d}{c}\right)i^{(1-c)/2}\exp\left(\frac{\pi i}{12}(c(a+d)-(c^2-1)bd-3) \right),
& c \equiv 1 \mod 2, c>0.
\end{cases}
 \]
\end{satz}

\begin{koro}\label{koro:m_ell_2_trafo}
We have $m_{\ell, 2}(\tau)=\ell^s m_\ell\left(\frac{\tau}{\ell}\right)^{-1}$ as well as 
$w_\ell(m_{\ell, 2}(\tau))=m_\ell(\ell\tau)$.
\end{koro}
\begin{proof}
We compute
\begin{align}
m_{\ell, 2}(\tau)=& m_\ell(S_0\tau)=\ell^s\left(\frac{\eta(\ell S_0\tau)}{\eta(S_0\tau)}\right)^{2s}
=\ell^s\left(\frac{\eta\left(S_0\frac{\tau}{\ell}\right)}{\eta(S_0\tau)}\right)^{2s}
=\ell^s\left(\frac{-i\sqrt{i\frac{\tau}{\ell}}\cdot\eta\left(\frac{\tau}{\ell}\right)}{-i\sqrt{i\tau}\cdot\eta(\tau)}\right)^{2s}\nonumber\\ 
=&\left(\frac{\eta\left(\frac{\tau}{\ell}\right)}{\eta(\tau)} \right)^{2s}
=\ell^s m_\ell\left(\frac{\tau}{\ell}\right)^{-1}\label{eq:ml_S0}
 \end{align}
and
\begin{align}
w_\ell(m_{\ell, 2}(\tau))=
\ell^s\left(\frac{\eta\left(\ell S_0\left(
\begin{smallmatrix} 0 & -1 \\ \ell & 0 \end{smallmatrix}  
\right)\tau \right)}{\eta\left(S_0\left(
\begin{smallmatrix} 0 & -1 \\ \ell & 0 \end{smallmatrix}  
\right)\tau\right)} \right)^{2s}
=\ell^s\left(\frac{\eta(\ell^2\tau)}{\eta(\ell\tau)}\right)^{2s}
=m_\ell(\ell\tau).\label{eq:ml_S0_wl}
 \end{align}
\end{proof}
From the definition of the minimal polynomial $M_\ell$ via \cite[Lemma 2.11]{ELGS} we immediately see that the function $m_{\ell, 2}(\tau)=m_\ell(S_0\tau)$, which is a conjugate of $m_\ell(\tau)$, has the same minimal polynomial $M_\ell(X, j(\tau))$ over $\C(j(\tau))$. Hence, the minimal polynomial of $m_{\ell, 2}(\tau)$ over $\mathbf{A}_0(\Gamma_0(\ell))$ is precisely $M_{\ell, 2}(X)=\frac{M_\ell(X, j(\tau))}{X-m_\ell(\tau)}$. We can now prove the following statement:

\begin{lemma}\label{lem:sigma_prec}
Let $\sigma_{\ell, n}(q)$, $m_\ell(q)$, $m_{\ell, 2}(q)$ and $j(q)$ be computed up to precision
\[ 
\prec(\ell, n)=(\ell^2+\ell+1)v-1,
 \]
where $v=\frac{\ell-1}{\gcd(\ell-1, 12)}=\ord(m_\ell)$ holds. Then one can obtain a rational expression for $\sigma_{\ell, n}(q)$ in terms of the other three modular functions in a unique way.
\end{lemma}
\begin{proof}
For determining the necessary precision we proceed in two steps. We first determine bounds on the order of the coefficients $a_i$ from equation \eqref{eq:darst_sigma_1}, which will subsequently give rise to bounds on the required precision.\\

According to \eqref{eq:darst_sigma_1} one has
\begin{equation}\label{eq:darst_sigma_2}
\sigma_{\ell, n}(q)\underbrace{\frac{\partial M_{\ell, 2}}{\partial X}(m_{\ell, 2}(q))}_{=:N}=\sum_{i=0}^{\ell-1}a_im_{\ell, 2}(q)^i.
 \end{equation}
Using \eqref{eq:ml_S0} one obtains $\ord(m_\ell)=v$, $\ord(m_{\ell, 2})=-\frac{v}{\ell}$ and thus
\begin{align*}
\ord(N)=&\ord\left(\frac{\frac{\partial M_\ell}{\partial X}(m_{\ell, 2}, j)(m_{\ell, 2}-m_\ell)-M_\ell(m_{\ell, 2}, j)}{(m_{\ell, 2}-m_\ell)^2}\right)\\
=&\min\left(\ord\left(\frac{\partial M_\ell}{\partial X}(m_{\ell, 2}, j)\right)-\frac{v}{\ell}, \ord(M_\ell(m_{\ell, 2}, j))\right)+\frac{2v}{\ell}.
 \end{align*}
Considering the orders of the different terms of $M_\ell$ similarly to \cite[Lemma 2.12]{ELGS_Rechnungen} one obtains
\[ 
iv\leq (v-k)\ell+v \quad \Rightarrow \quad -\frac{iv}{\ell}-k \geq -v-\frac{v}{\ell}\quad \Rightarrow \quad \ord(M_\ell(m_{\ell, 2}, j))\geq -v-\frac{v}{\ell}
 \]
whenever the coefficient $a_{i, k}$ in the polynomial $M_\ell$ does not vanish. In the same way one gleans $\ord\left(\frac{\partial M_\ell}{\partial X}(m_{\ell, 2}, j)\right)\geq -v$,
which yields $\ord(N)\geq-v+\frac{v}{\ell}$.\\
Applying $w_\ell$ to \eqref{eq:darst_sigma_2} and using \eqref{eq:ml_S0_wl} one obtains $\ord(m_{\ell, 2}^*)=v\ell$, which similarly implies
$\ord(M_\ell(m_{\ell, 2}^*, j^*))\geq 0$, 
$\ord\left(\frac{\partial M_\ell}{\partial X}(m_{\ell, 2}^*, j^*)\right)\geq -v\ell$
and hence $\ord(N^*)\geq v(1-\ell)$.\\
Furthermore, using the definition of $\sigma_{\ell, n}$, formulae \eqref{eq:x_def}, \eqref{eq:y_def}, \eqref{eq:eta_def} and lemma \ref{lem:wp_trafo_allgemein} one can calculate
\begin{align*}
\ord(\sigma_{\ell, n})=&\ord(G_{\ell, n})+\ord(H_{\ell, n})+r\ord(p_1)-\ord(\Delta)
\geq 0 +\frac{1}{\ell}+0-1=-1+\frac{1}{\ell},\\
\ord(\sigma_{\ell, n}^{\ast})=&\ord(G_{\ell, n}^{*})+\ord(H_{\ell, n}^*)+r\ord(p_1^*)-\ord(\Delta^*)
\geq 1 +0+0-\ell=1-\ell.
 \end{align*}
Altogether, this yields
\[ 
\ord(\sigma_{\ell, n}N)\geq -v-1+\frac{v+1}{\ell},\quad 
\ord(\sigma_{\ell, n}^*N^*)\geq 1-\ell+v(1-\ell).
 \]
From equation \eqref{eq:darst_sigma_2} one obtains
\begin{align}
\ord(a_i)\geq&\ord(\sigma_{\ell, n}Nm_{\ell, 2}^{-i})\geq -v-1+\frac{v+1}{\ell}+\frac{iv}{\ell}\quad \Rightarrow\quad \ord(a_i)\geq -v,\nonumber\\
\ord(a_i^*)\geq& \ord(\sigma_{\ell, n}^*N^*(m_{\ell, 2}^*)^{-i})\geq 1-\ell+v(1-\ell)-iv\ell \geq v+1-v\ell-\ell-(\ell-1)v\ell\nonumber\\
=&-\ell^2v-\ell+v+1.\label{eq:ord_ai_stern}
\end{align}
We next consider the equation
\begin{equation}\label{eq:darst_sigma_3}
a_i(q)\underbrace{\frac{\partial M_\ell}{\partial Y}(m_\ell(q), j(q))}_{=:M} =\sum_{i=i_1}^{i_2}\sum_{k=0}^{v-1}b_{i, k}m_\ell(q)^ij(q)^k
 \end{equation}
according to corollary \ref{koro:darstellung_besser}. Again proceeding similarly to \cite[Lemma 2.12]{ELGS_Rechnungen} one easily shows $\ord(M)\geq 1$, $\ord(M^*)\geq -(v-1)\ell-v$, which yields
\[ 
\ord(a_iM)\geq 1-v,\quad \ord(a_i^*M^*)\geq -\ell^2v-\ell+v+1-(v-1)\ell-v=-(\ell^2+\ell)v+1=:o.
 \]
From this we deduce the bounds
\[ 
iv-k\geq 1-v\quad \text{and}\quad -iv-\ell k \geq o \quad \Rightarrow\quad iv-k\leq iv+\ell k \leq -o,
 \]
for the orders $iv-k=\ord(m_\ell^ij^k)$ of the summands on the right hand side in \eqref{eq:darst_sigma_3}. Taking their difference finally yields the necessary precision $\prec(\ell, n)=(\ell^2+\ell+1)v-1$ we claimed.
\end{proof}
\begin{bem}
The estimate \eqref{eq:ord_ai_stern} can obviously be improved significantly for small values of $i$. E.~g., for $i=0$ one obtains the bound $\ord(a_0^*)\geq v+1-(v+1)\ell$, whence the precision $\prec(\ell, n)=(2\ell+1)v-1$ is sufficient to calculate the representation for $a_0$.
\end{bem}

Using the approach sketched in lemma \ref{lem:sigma_prec} we thus obtain the equation 
\begin{equation}\label{eq:sigma_darst_konkret}
\sigma_{\ell, n, \chi}(\tau)=R(j(\tau), m_\ell(\tau), m_\ell(S_0\tau))
 \end{equation}
as a special case of \eqref{eq:sigma_darst_allgemein}.
The following lemma provides equivalent statements of \eqref{eq:sigma_darst_konkret}.

\begin{lemma}
For $0\leq k<\ell$ let $T_k=\left(\begin{smallmatrix}
1 & k \\ 0 & 1
\end{smallmatrix}\right) \in \Gamma_0(\ell)$ be the matrices introduced in equation \eqref{eq:def_matrix_T} and let $S_k=\left(\begin{smallmatrix}
0 & -1 \\ 1 & k
\end{smallmatrix}\right)$ be as in equation \eqref{eq:def_rep_Sk}. Furthermore, let $P=(x(\zeta_\ell, q), y(\zeta_\ell, q))$ and $Q_k=(x(\zeta_\ell^kq^{\frac{1}{\ell}}, q), y(\zeta_\ell^kq^{\frac{1}{\ell}}, q)), 0 \leq k < \ell$, denote these $\ell$-torsion points on the Tate curve. Then the action of $T_k$ transforms equation \eqref{eq:sigma_darst_konkret} into
\begin{equation}\label{eq:sigma_darst_konkret_trafo}
\frac{G_{\ell, n, \chi}(q)\left(\sum_{\lambda \in \F_\ell^*}\chi(\lambda)(\lambda Q_k)_V\right) p_1(q)^rG_{\chi^{-1}}(P, Q_k)}{\Delta(q)}=R(j(\tau), m_\ell(\tau), m_\ell(S_k\tau)),
 \end{equation}
where $V=x$ holds for $n$ odd and $V=y$ for $n$ even.

\end{lemma}
\begin{proof}
We first consider the right hand side of the equation. Since $T_k$ lies in $\Gamma_0(\ell)$, both $m_\ell(\tau)$ and $j(\tau)$ are invariant under the action of $T_k$. Furthermore,
\[ 
\begin{pmatrix}
0 & -1 \\ 1 & 0
\end{pmatrix}
\begin{pmatrix}
1 & k \\ 0 & 1
\end{pmatrix}
=
\begin{pmatrix}
0 & -1 \\ 1 & k
\end{pmatrix}
 \]
and hence $\left.m_\ell(S_0\tau)\right|_{T_k}=m_\ell(S_0T_k\tau)=m_\ell(S_k\tau)$ holds.\\
On the left hand side we examine the different components of $\sigma_{\ell, n, \chi}(\tau)$ according to corollary \ref{koro:gs_atkin_def}. We first see $\left.\Delta(q)\right|_{T_k}=\Delta(q)$, $\left.p_1(q)\right|_{T_k}=p_1(q)$, and by the proof of corollary \ref{koro:gs_atkin_def} $\left.G_{\ell, n, \chi}(q)\right|_{T_k}=\chi^{-1}(1)G_{\ell, n, \chi}(q)=G_{\ell, n, \chi}(q)$ holds. Furthermore, using $c=\frac{1}{(2\pi i)^2}$ from lemma \ref{lem:wp_trafo_allgemein} one calculates
\[ 
\left.x\left(q^{\frac{\lambda}{\ell}}, q\right)\right|_{T_k}
=\left.c\wp\left(\frac{\lambda\tau}{\ell}, \tau\right)\right|_{T_k}
=c\wp\left(\frac{\lambda(\tau+k)}{\ell}, \tau\right)
=x\left(\zeta_\ell^{k\lambda}q^{\frac{\lambda}{\ell}}, q\right)
\]
and an analogue statement holds for $y(q^{\frac{\lambda}{\ell}}, q)$. This implies
\[ 
\left.H_{\ell, n, \chi}(q)\right|_{T_k}=\sum_{\lambda \in \F_\ell^*}\chi(\lambda)V\left(\zeta_\ell^{k\lambda}q^{\frac{\lambda}{\ell}}, q\right)
=\sum_{\lambda \in \F_\ell^*}\chi(\lambda)(\lambda Q_k)_V,
 \] 
where $V=x$ holds for odd $n$ and $V=y$ for even $n$. In addition, we deduce
\[ 
\left.G_{\chi^{-1}}(e_\ell(P, Q_0))\right|_{T_k}=G_{\chi^{-1}}(e_\ell(P, Q_k)).
 \]
\end{proof}

\subsection{Application}
Let $\ell$ be an Atkin prime for an elliptic curve $E$ over $\F_p$. The theorems from \cite[pp.~236--239]{Schoof} concerning the decomposition of the modular polynomial $\Phi_\ell$ and the constructions from \cite{Mueller} imply that for some $r \mid \ell+1$, $r>1$, there exists an $\ell$-isogeny $\psi: E \rightarrow E^\prime$ defined over $\F_{p^r}$ and which corresponds to a root $m_\ell(E) \in \F_{p^r}\backslash\F_{p^{r-1}}$ of $M_\ell(X, j(E))$. Henceforth, we assume that $r>2$ holds.
For $n \mid \ell-1$, a character $\chi$ of order $n$ and a point $Q$ on $E$ we use the notation
\[ 
G_{\ell, n, \chi}(E, Q):=\sum_{a=1}^{\ell-1}\chi(a)(aQ)_V,
 \]
where $V=x$ holds for odd and $V=y$ for even $n$.\\

We can specialize the formulae \eqref{eq:sigma_darst_konkret} and \eqref{eq:sigma_darst_konkret_trafo} to a concrete elliptic curve $E/\F_p$ in question using the same considerations as in \cite[p.~14]{ELGS}. By means of the Deuring lifting theorem from \cite[p.~184]{Lang} we can lift the curve $E/\F_p$ in question to a curve $E_0$ over a number field $K$. This means there is a prime ideal $\mathfrak{P} \subset \OO_K$ with residue field $\F_p$ such that the reduction of $E_0$ modulo $\mathfrak{P}$ is a non-singular elliptic curve which is isomorphic to $E$. In particular, one sees that for the value of $\tau$ corresponding to the curve $E_0$ the quantity $m_\ell(E)$ corresponds to the value $m_\ell(\tau)$ in \eqref{eq:sigma_darst_konkret_trafo}. Since the values $\varphi_p(m_\ell(E))$ and $\varphi_p^2(m_\ell(E))$ are also roots of $M_\ell(X, j(E))$, each of them likewise corresponds to a conjugate $m_\ell(S_k\tau)$ for a suitable $k$ with $0\leq k<\ell$. Here we use the fact that we have assumed $r>2$ and that hence $\varphi_p^2(m_\ell(E))\neq m_\ell(E)$ holds.\\
Furthermore, by applying $\varphi_p$ to the formulae for computing the Elkies factor $F_{\ell, \lambda}$ from \cite{Schoof} and \cite[pp.~89--106]{Mueller} it follows that the values $\varphi_p(m_\ell(E))$, $\varphi_p^2(m_\ell(E))$ correspond to the isogenies $\phi_p(\psi): E \rightarrow \phi_p(E^\prime)$ as well as $\phi_p^2(\psi): E \rightarrow \phi_p^2(E^\prime)$. Denoting by $P\neq \OO$ a point in $\ker(\psi)$, we deduce $\phi_p(P) \in \ker(\phi_p(\psi))$ and $\phi_p^2(P) \in \ker(\phi_p^2(\psi))$.
Specialising formula \eqref{eq:sigma_darst_konkret_trafo} to the curve $E$ in question thus yields the equations
\begin{align*}
R(j(E), m_\ell(E), \varphi_p(m_\ell(E)))&=\frac{G_{\ell, n, \chi}(E, P)G_{\ell, n, \chi}(E, \phi_p(P))p_1(E)^rG_{\chi^{-1}}(e_\ell(P, \phi_p(P)))}{\Delta(E)},\\
R(j(E), m_\ell(E), \varphi_p^2(m_\ell(E)))&=\frac{G_{\ell, n, \chi}(E, P)G_{\ell, n, \chi}(E, \phi_p^2(P))p_1(E)^rG_{\chi^{-1}}(e_\ell(P, \phi_p^2(P)))}{\Delta(E)}.
 \end{align*}
From this we directly glean
\begin{align}
\frac{R(j(E), m_\ell(E), \varphi_p(m_\ell(E)))}{R(j(E), m_\ell(E), \varphi_p^2(m_\ell(E)))}
=&\frac{G_{\ell, n, \chi}(E, \phi_p(P)) G_{\chi^{-1}}(e_\ell( P, \phi_p(P))) }{G_{\ell, n, \chi}(E, \phi_p^2(P)) G_{\chi^{-1}}(e_\ell( P, \phi_p^2(P)))}.\label{eq:atkin_gl_1}
\end{align}
Furthermore, for $p \equiv 1 \mod n$ we calculate
\begin{equation}\label{eq:gs_phi_p}
G_{\ell, n, \chi}(E, \phi_p(P))=\sum_{a=1}^{\ell-1}\chi(a)(a\phi_p(P))_V=\sum_{a=1}^{\ell-1}\chi^p(a)(aP)_V)^p
=G_{\ell, n, \chi}(E, P)^p,
 \end{equation}
which immediately implies $G_{\ell, n, \chi}(E, \phi_p^2(P))=G_{\ell, n, \chi}(E, P)^{p^2}$. Hence, we obtain
\begin{equation}
\frac{G_{\ell, n, \chi}(E, \phi_p(P))}{G_{\ell, n, \chi}(E, \phi_p^2(P))}
=G_{\ell, n, \chi}(E, P)^{p-p^2}=(G_{\ell, n, \chi}(E, P)^n)^{\frac{p(1-p)}{n}}\label{eq:atkin_gl_2},
\end{equation}
since $n \mid 1-p$ holds. Using equation \eqref{eq:atkin_gl_1} one deduces
\begin{equation}\label{eq:atkin_gl_3}
\frac{R(j(E), m_\ell(E), \varphi_p(m_\ell(E)))}{R(j(E), m_\ell(E), \varphi_p^2(m_\ell(E)))}(G_{\ell, n, \chi}(E, P)^n)^{\frac{p(p-1)}{n}}
=\frac{G_{\chi^{-1}}(e_\ell( P, \phi_p(P)))}{G_{\chi^{-1}}(e_\ell( P, \phi_p^2(P)))}.
 \end{equation}
As detailed in \cite{ELGS, ELGS_Rechnungen}, the value of $G_{\ell, n, \chi}(E, P)^n$ can be computed from the values $j(E)$ and $m_\ell(E)$ by means of the precomputed universal elliptic Gau{\ss} sums for the Elkies case. Hence, the left hand side of equation \eqref{eq:atkin_gl_3} can be computed using the precomputed rational expressions.\\

According to equation \eqref{eq:char_gl}
\[ 
\phi_p^2(P)=t\phi_p(P)-pP
 \]
holds. Using the properties of the Weil pairing $e_\ell$ this implies
\[ 
e_\ell(P, \phi_p^2(P))=e_\ell(P, t\phi_p(P)-pP)=e_\ell(P, \phi_p(P))^t.
 \]
One easily sees $G_{\chi^{-1}}(\xi^t)=\chi(t)G_{\chi^{-1}}(\xi)$, which yields
\[ 
\frac{G_{\chi^{-1}}(e_\ell( P, \phi_p(P)))}{G_{\chi^{-1}}(e_\ell( P, \phi_p^2(P)))}
=\frac{G_{\chi^{-1}}(e_\ell( P, \phi_p(P)))}{\chi(t)G_{\chi^{-1}}(e_\ell( P, \phi_p(P)))}
=\chi^{-1}(t)
 \]
for the right hand side of equation \eqref{eq:atkin_gl_3}. Hence, equation \eqref{eq:atkin_gl_3} allows to compute the index of $t \in (\Z/\ell\Z)^*$ modulo $n$. If the computation is performed for appropriate different coprime divisors $n$ of $\ell-1$ one obtains the value $t$ modulo $\ell$ as required in Schoof's algorithm.\\

Finally, we consider the assumptions we have made when deriving the formulae. Firstly, we assumed $r>2$. If $r=2$ holds, no further calculations are necessary for determining $t$ modulo $\ell$, though: Namely, according to \cite[p.~236]{Schoof} for $P \in E[\ell]$ we have $\phi_p^2(P)=aP$ for some $a \in \F_\ell^*$. Introducing this in equation \eqref{eq:char_gl} yields
\[ 
(a+p)P=t\phi_p(P).
 \]
Since $\ell$ is an Atkin prime the left hand side has to vanish. Thus, one obtains the congruences $a\equiv-p \mod \ell$ and $t\equiv 0 \mod \ell$.\\

In addition, in equations \eqref{eq:gs_phi_p} and \eqref{eq:atkin_gl_2} we assumed that $p \equiv 1 \mod n$ holds. In general, let us denote by $p^\prime$ the inverse modulo $n$ of the prime $p$ in question and let us write $p=nq_1+m_1, p^2=nq_2+m_2$ with $0\leq m_1, m_2<n$. Now instead of equation \eqref{eq:gs_phi_p} we glean
\[ 
G_{\ell, n, \chi}(E, \phi_p(P))=\sum_{a=1}^{\ell-1}\chi(a)(a\phi_p(P))_V=\sum_{a=1}^{\ell-1}\chi^{p^\prime p}(a)(aP)_V)^p
=G_{\ell, n, \chi^{p^\prime}}(E, P)^p
 \]
and thence $G_{\ell, n, \chi}(E, \phi_p^2(P))=G_{\ell, n, \chi^{p^{\prime 2}}}(E, P)^{p^2}$. This yields 
\[ 
\frac{G_{\ell, n, \chi}(E, \phi_p(P))}{G_{\ell, n, \chi}(E, \phi_p^2(P))}=
\frac{(G_{\ell, n, \chi^{p^\prime}}(E, P)^n)^{q_1} G_{\ell, n, \chi^{p^\prime}}(E, P)^{m_1} G_{\ell, n, \chi^{-1}}(E, P) }{(G_{\ell, n, \chi^{p^{\prime 2}}}(E, P)^n)^{q_2} G_{\ell, n, \chi^{p^{\prime 2}}}(E, P)^{m_2} G_{\ell, n, \chi^{-1}}(E, P)}
 \]
instead of equation \eqref{eq:atkin_gl_2}. The value $G_{\ell, n, \chi^{p^\prime}}(E, P)^n$ can again be determined by means of the universal elliptic Gau{\ss} sums for the Elkies case. Furthermore, since $p^\prime m_1 \equiv p^\prime p\equiv 1 \mod n$ holds, the second quantity $G_{\ell, n, \chi^{p^\prime}}(E, P)^{m_1} G_{\ell, n, \chi^{-1}}(E, P)$ can be directly computed by means of the universal elliptic Jacobi sum $J_{\ell, n, \chi^{p^\prime}, m_1}$ from \cite[Lemma 3.1]{ELGS_Rechnungen}. An analogue statement holds for the denominator. Thus, when considering arbitrary primes $p$, equation \eqref{eq:atkin_gl_1} can still be transformed into an equation analogous to \eqref{eq:atkin_gl_3}, whose left hand side can be computed by means of precomputed rational expressions.

\subsection{Run-time}
We confine ourselves to considering the run-time required for the evaluation of the expression $R(j(E), m_\ell(E), \varphi_p(m_\ell(E)))$. Combining lemma \ref{lem:darst_sigma_allgemein} and corollary \ref{koro:darstellung_besser} yields
\begin{equation}\label{eq:darst_sigma_gesamt}
\sigma_{\ell, n}(E)=\frac{\sum_{i_1=0}^{\ell-1}\varphi_p(m_{\ell}(E))^{i_1}\sum_{i_2=n_0}^{n_1}m_\ell(E)^{i_2}\sum_{i_3=0}^{v-1}a_{i_1, i_2, i_3}j(E)^{i_3}}{\frac{\partial M_{\ell, 2}}{\partial X}(\varphi_p(m_{\ell}(E)))\frac{\partial M_\ell}{\partial Y}(m_\ell(E), j(E))}.
 \end{equation}
From lemma \ref{lem:sigma_prec} and equation \eqref{eq:darst_sigma_3} we deduce that on average the computation of the term depending on $m_\ell$ and $j$ for fixed $i_1$ requires $O(\ell^2v)$ multiplications in $\F_{p^r}[\zeta_n]$. Summing over $i_1$ subsequently we see the total cost amounts to $O(\ell^3v)$ multiplications in $\F_{p^r}[\zeta_n]$. 
It is apparent that when using values of $\ell \in O(\log p)$ the run-time of this step alone for 
\textit{one} fixed $n$ considerably exceeds the run-time of $\tilde{O}(\ell^2\log p)$ multiplications in $\F_p$ necessary for fixed $\ell$ in Schoof's original algorithm, let alone the run-time $\tilde{O}(\ell\log p)$ required in the Elkies case or the ones corresponding to its numerous improvements. Taking into account that $\ell \in O(v)$ holds asymptotically according to \cite{Abramovich} we notice the alternative for the Atkin case we presented exhibits a run-time which makes its use for counting points infeasible.\\
Obviously this would still be the case if we had chosen another modular function instead of $m_\ell$ for deriving our rational expression. Indeed, in this case the value of $v$ would be smaller. However, the factor $\ell^3$ would still cause a run-time significantly exceeding that of existing approaches.

\footnotesize{
\bibliography{lit}}

\end{document}